\documentclass[reqno]{amsart}

\usepackage{amsmath, amsfonts, amsthm, amssymb, color,  graphicx, mathrsfs, cite}
\usepackage{comment,stmaryrd}

\theoremstyle{plain}
\newtheorem{theorem}{Theorem}[section]
\newtheorem{lemma}[theorem]{Lemma}
\newtheorem{proposition}[theorem]{Proposition}
\newtheorem{corollary}[theorem]{Corollary}
\newtheorem{definition}[theorem]{Definition}

\newtheorem{example}[theorem]{Example}

\numberwithin{equation}{section}
\allowdisplaybreaks

\usepackage{ifpdf}
\ifpdf \usepackage[colorlinks=true, citecolor=blue, linkcolor=blue, urlcolor=blue]{hyperref} \fi

\def\thrm{\begin{theorem}}
\def\thrml#1{\begin{theorem}\label{#1}}
\def\ethrm{\end{theorem}}
\def\lmm{\begin{lemma}}
\def\lmml#1{\begin{lemma}\label{#1}}
\def\elmm{\end{lemma}}
\def\dfntn{\begin{definition}}
\def\dfntnl#1{\begin{definition}\label{#1}}
\def\edfntn{\end{definition}}
\def\crllr{\begin{corollary}}
\def\crllrl#1{\begin{corollary}\label{#1}}
\def\ecrllr{\end{corollary}}
\def\xmpl{\begin{example}}
\def\xmpll#1{\begin{example}\label{#1}}
\def\exmpl{\end{example}}
\def\nmrt{\begin{enumerate}}
\def\enmrt{\end{enumerate}}
\def\qtn{\begin{equation}}
\def\qtnl#1{\begin{equation}\label{#1}}
\def\eqtn{\end{equation}}
\def\prpstn{\begin{proposition}}
\def\prpstnl#1{\begin{proposition}\label{#1}}
\def\eprpstn{\end{proposition}}

\def\proof{{\bf Proof}.\ }
\def\eprf{\hfill$\square$}

\DeclareMathOperator{\aut}{Aut}

\def\lg{\langle}
\def\rg{\rangle}

\def\irr{{\rm Irr}}

\def\proof{{\bf Proof}.\ }
\def\eprf{\hfill$\square$}

\usepackage{geometry}
\begin{document}

\title{A counterexample to a conjecture of Moret\'o }
\maketitle

\begin{center}
{\author Gang Chen$^{a, *}$,  Wenhua Zhao$^{b}$}
\end{center}

\vskip 3mm

\begin{abstract} As an extension of the It$\hat{o}$-Michler theorem to the nonabelian or nonnormal  Sylow subgroups, in \cite[Conjecture ~4.4]{Mo1} A. Moret\'o raised the following conjecture: let $G$ be a finite group and $P$  a Sylow $p$-subgroup of $G$. Then the number of irreducible characters of $G$ whose degrees are divisible by $p$ is not less than the length of any subgroup chain between $N_G(P)$ and $G$. In this short note, we give a counterexample to this conjecture.

\vskip 2mm

{Keywords: It$\hat{o}$-Michler theorem, Irreducible character, Counterexample }
\end{abstract}

\small {2020 Mathematics Subject Classification:  20C15; 20D20.}
\date{}


\renewcommand{\thefootnote}{\empty}
\footnotetext{$^{a, *}$ Corresponding author,  Email: {chengangmath@mail.ccnu.edu.cn},  School of Mathematics and Statistics,  Hainan University, Haikou 570228, China}
\footnotetext{ $^{b}$ Email:  {zwh1999@mails.ccnu.edu.cn},  School of Mathematics and Statistics, Central China Normal University, Wuhan 430079, China.}

\section{introduction}

Throughout this note, $G$ is a finite group, $p$ is a prime, and $\irr(G)$ is the set of complex irreducible characters of $G$. The set of irreducible characters of $G$  whose degrees are divisible by $p$ is denoted by $\irr_p(G)$. Let $P$ be a Sylow $p$-subgroup of $G$.  The chain length of the longest strictly ascending subgroup chain between $N_G(P)$ and $G$ is denoted by $l_p(G)$.

\medskip

The following is the well-known It$\hat{o}$-Michler theorem.

\begin{theorem}[{\rm It$\hat{o}$-Michler}]\label{1128u}
	Let $G$ be a finite group. Then $G$ has a normal abelian Sylow $p$-subgroup iff $|\irr_p(G)|=0$.
	
\end{theorem}

\medskip

As an extension of the It$\hat{o}$-Michler theorem to the nonabelian or nonnormal Sylow subgroups, A. Moret\'o raised the following conjecture; see \cite[Conjecture 4.4]{Mo1}.

\medskip

{\bf Conjecture.}\, {\it Let $G$ be a finite group, $p$ be a prime, and $P$ a Sylow $p$-subgroup of $G$. If $G$ has $n$ irreducible characters of degree divisible by $p$, then the length of any subgroup chain between $N_G(P)$ and $G$ is at most $n$. }

\medskip

It was proved that if $|\irr_p(G)|=1$, then $N_G(P)$ is maximal in $G$ in \cite{GGLMNT}. Hence, the above conjecture holds in this case. Our starting point is to consider this conjecture where $|\irr_p(G)|=2$ and the counterexample was found in this case.  The goal of this short note is to give a counterexample to this conjecture with the help of OpenAI's ChatGPT (GPT-6).

\medskip 

In the upcoming manuscript of the authors, it will be proved that the conjecture has an affirmative answer if $G$ is supersolvable or $G$ is a solvable group with $|\irr_p(G)|=2$ and Fitting height not more than $2$.

\medskip

\section{counterexample and proofs}

Our strategy in the proof can be described as follows. 

Let 
$$
1=N_0\vartriangleleft N_1\vartriangleleft N_2\vartriangleleft \cdots \vartriangleleft N_{r-1}\vartriangleleft N_r=G
$$
be a normal subgroup chain of $G$. The chain length of the longest strictly ascending subgroup chain between $N_G(P)N_i$ and $N_G(P)N_{i+1}$ is denoted by $l(N_G(P)N_i,N_G(P)N_{i+1})$.  One can easily check that
$$
l_p(G)=\sum\limits_{i=0}^{r-1}l(N_G(P)N_i,N_G(P)N_{i+1}).
$$
If we could prove that 
\begin{equation}\label{1359}
l(N_G(P)N_i,N_G(P)N_{i+1})\leqslant |\irr_p(G/N_{i}|/N_{i+1}/N_i)|, 
\end{equation}
for $i=1, \cdots, r-1$, then the original conjecture could be true. 

\medskip

 However, as we shall see that \eqref{1359} does not hold always. 
Suppose $M$ and $N$ are normal subgroups of $G$ satisfying $M\subseteq N$ and 
$$
l(N_G(P)M,N_G(P)N)>|\irr_p(G/M|G/N)|
$$
then  $N/M$ is called a \textbf{bad factor} of $G$ with respect to  $p$. In our construction, $|\irr_p(G)|=2$, $l_p(G)=3$, and $G$ has two bad factors which are provided by distinct factors in one normal subgroup chain of $G$. 

\medskip

In the rest part of this note,  let $F=GF(8)$ be the Galois field of order $8$, whose multiplicative group is denoted by $F^\times$. It is well known that the map $\sigma$ defined by  $\alpha\mapsto \alpha^2$ is an automorphism of $F$. Set $\tau=\sigma^2$, which is equal to $\sigma^{-1}$.
Let $U=\{(a,b)|a\in F, b\in F\}$.  For each $(a,b),(c,d)\in U$,  the multiplication of $U$ is defined in the following way:

\begin{equation}\label{1545r}
(a,b)(c,d)=(a+c,b+d+a^\tau c)
\end{equation}

\medskip

\begin{lemma}\label{1534c}
	With respect to the multiplication defined in \eqref{1545r}, $U$ is a group.
	\end{lemma}
\proof Choose $(a,b), (c,d)$ and $(e,f)$ from $U$. We have the following: 	
\begin{enumerate}
  \item $[(a,b)(c,d)](e,f)=(a+c,b+d+a^\tau c)(e, f)=(a+c+e,b+d+f+a^\tau c+a^\tau e+c^\tau e)$,
  \medskip
  \item $(a,b)[(c,d)(e,f)]=(a,b)(c+e,d+f+c^\tau e)=(a+c+e,b+d+a^\tau c+a^\tau e+c^\tau e)$,
  \medskip
  \item $(0,0)(a,b)=(a,b)=(a,b)(0,0)$,
  \medskip
  \item $(a,b)(-a,-b+aa^\tau)=(0, 0)=(-a,-b+aa^\tau)(a,b)$.
\end{enumerate}

\medskip

The first two formulas show that the defined multiplication is associative, the third one shows that $(0, 0)$ is the identity element and the fourth one shows that each element in $U$ has an inverse. We conclude that $U$ is a group with respect to the multiplication defined in \eqref{1545r}.

\eprf

\medskip

For each $t\in F^{\times}$, we define a permutation on $U$ in the following way
\begin{equation}\label{1611q}
\varphi_t: U\longrightarrow U, \quad (a,b)\mapsto (ta,t^5b).
\end{equation}
In addition,  the natural extension of $\sigma$ to $U$, also denoted by $\sigma$, is defined in the following way:
\begin{equation}\label{1614t}
\sigma: U\longrightarrow U, \quad (a,b)\mapsto (a^\sigma,b^\sigma)
\end{equation}

\medskip

\begin{lemma}\label{1621y} Keep the notation above. Then $\sigma$ and $\varphi_t$ for each $t\in F^\times$ are automorphisms of the group $U$. In addition, $\sigma^{-1}\varphi_t\sigma=\varphi_{t^\tau}$ for each $t\in F^\times$.

\end{lemma}
	
	\proof For each $(a, b), (c, d)\in U$, we have that

\begin{enumerate}
  \item $\varphi_{t}[(a,b)(c,d)]=(ta+tc,t^5b+t^5d+t^5a^4c)=\varphi_{t}(a,b)\varphi_t(c,d)$,

  \medskip

  \item $\sigma[(a,b)(c,d)]=(a^2+c^2,b^2+d^2+(a^4)^2c^2)=(a^2,b^2)(c^2,d^2)=\sigma(a,b)\sigma(c,d)$,

  \medskip

  \item $(\sigma^{-1}\varphi_{t}\sigma)(a,b)=\sigma^{-1}\varphi_t(a^2,b^2)=\sigma^{-1}(ta^2,t^5b^2)=(t^4a,t^{20}b)=\varphi_{t^\tau}(a,b)$.
\end{enumerate}
These equalities prove the statements in the lemma.
\eprf

\medskip

Let  $p=3$ and $t$ be a generator of $F^\times$.  Set $T=\lg \varphi_{t}\rg$, $P=\lg \sigma\rg$ as  elements of $\aut(U)$, and $K=T\rtimes P$.  By Lemma~ \eqref{1621y}, one can see that  $K$ is a subgroup of $\aut(U)$.  Finally,  take $G=U\rtimes K$.  It is easy to see that $|G|=2^6\cdot 3\cdot 7$ and $P$ is a Sylow $p$-subgroup of $G$.

\medskip

\begin{theorem}\label{1702t} Keep the notation above. Then $|\irr_p(G)|=2$ and  $l_p(G)= 3$.

	\end{theorem}
\proof The proof of the theorem is divided into several steps.

\medskip

\textbf{Step 1.}\,  Let $Z=Z(U)$. Then $|\irr_p(G/Z|U/Z)|=0$.

\medskip

It is easy to see that
$$
\{(0,d)|d\in F\}\subseteq Z.
$$
 Also, by formula \eqref{1545r} we have that $(c,d)=(c,0)(0,d)$ for each $c, d\in F$.  Thus,  $(a,b)\in Z$ iff
 $(a,b)(c,0)=(c,0)(a,b)$ for each $c\in F$, which is equivalent to
 $$
 a^\tau c=c^\tau a.
 $$
 If $a\ne 0$, then we have that
 $$
 (\frac{c}{a})^\tau=\frac{c}{a},
 $$
 which is true for any $c\in F$. We get a contradiction as $\tau$ is a nonidentity automorphism of $F$.
 It then follows that  $a=0$, and therefore
 \begin{equation}\label{957c}
  Z=\{(0,d)|d\in F\}.
 \end{equation}

 \medskip

 It is easy to check that $U/Z\cong F^+$, where $F^+$ is the additive group of the field $F$.  Thus,
 \begin{equation}\label{1051t}
 G/Z\cong F^+\rtimes (\lg \varphi_{t}\rg \rtimes \lg \sigma\rg),
 \end{equation}
  where $\varphi_{t}(a)=ta$, and $\sigma(a)=a^2$ for each $a\in F$.

  \medskip

  To complete the proof, it suffices to show that $|\irr_p(H|F^+)|=0$, where $H=F^+\rtimes (\lg \varphi_{t}\rg \rtimes \lg \sigma\rg)$.

\medskip

For each $0\ne a\in F$, one can easily check that
\begin{equation}\label{1021t}
C_K(a)=\{\sigma\varphi_{a^3}, \sigma^2\varphi_a, {\rm id} \},
\end{equation}
which is a Sylow 3-subgroup $P_a$ of $H$.

\medskip

Let $\theta\in \irr(F^+)$ be a nonprincipal irreducible character. For $k\in K$, $\theta^k=\theta$ iff $\theta^k(a)=\theta(a)$ for some generator $a$ of $F^{\times}$ iff  $k\in C_K(a)$ iff $k\in P_a$.

\medskip

It then follows that
 $$
 I_H(\theta)=F^+\rtimes P_a.
 $$
Since $\gcd(|F^+|,|P_a|)=1$, by \cite[Corollary 11.22]{I1} we have that  $\theta$ is extendible to $I_H(\theta)$. Let $\hat{\theta}$ be such an extension. Then by \cite[Corollary 6.17]{I1} we have that
$$
\irr(I_H(\theta)|\theta)=\{\hat{\theta}\xi|\xi\in \irr(P_a)\}.
$$
Then by Clifford correspondence, we have that each irreducible character in $\irr(H|\theta)$ has degree $7$ as $|H: I_H(\theta)|=7$.

\medskip

 We then conclude that
$$
|\irr_p(H|F^+)|\leqslant\sum\limits_{\theta\in \irr(F^+)-1_{F^+}}|\irr_p(H|\theta)|=0.
$$

\medskip

\textbf{Step 2.}\,  $|\irr_p(G|Z)|=0$.

Similarly,  for each $x=(0, d)\in Z$ with $d\ne 0$, we have
\begin{equation}\label{1412t}
C_K(x)=\{\sigma \varphi_{x^2}, \sigma^2\varphi_{x^3}, {\rm id} \},
\end{equation}
which is a Sylow $p$-subgroup of $G$.

\medskip 

For each $a,c\in F^{\times}$, one can easily get that the commutator
\begin{equation}\label{2304c}
	[(a, 0), (ac, 0)]=(0, a^5c+a^5c^4)=(0,a^5(c+c^4)).
	\end{equation}

\medskip

Let $f(X)=X^3+X+1\in F_2[X]$.   Since $f=X^3+X+1$ has no root in $F_2$, we have that $f$ is irreducible in $F_2$. It then follows that $F$ is a splitting field for $f$ over $F_2$ and $F=F_2[c_0]$ for a root $c_0$ of $f$ in $F$.

\medskip

Note that
$$
c_0+c_0^4=c_0(c_0^3+c_0+1)-c_0^2=c_0^2.
$$
Note that $x\mapsto  x^5$ is a permutation on $F^\times$. Thus, there exists a unique  $b\in F^\times $ such that  $b^5=c_0^2$. We have
\begin{equation}
	[(a, 0), (ac_0, 0)]=(0,a^5(c_0+c_0^4))=(0,a^5c_0^2)=(0,(ab)^5).
\end{equation}
and hence
$$
Z=\{(0, a^5)|a\in F^{\times}\}\cup\{(0,0)\}\subseteq U'.
$$
Since $U/Z\cong F^+$,  $U'\subseteq Z$.  We conclude that  $U'=Z$.

\medskip

Deducing as before, we obtain that  for any $\theta\in \irr(Z)-1_Z$, $I_K(\theta)=Q_\theta$ for some Sylow $p$-subgroup $Q_\theta$ of $G$. Thus,
$$
I_G(\theta)=I_K(\theta)U=U\rtimes Q_\theta.
$$


 \medskip

 Now let  $N=\ker(\theta)$. Then one can easily see that
 $$
 |N|=4, \quad {\rm and}\quad  (U/N)'=Z/N.
 $$
 Thus $U/N$ has $8$ linear characters and none of them lies over $\theta$.  In addition,
 $U/N$ has two nonlinear irreducible characters both of which have degree $2$. Recall that  $\irr(Z/N)=\{1_{Z/N}, \theta\}$, thus these two  nonlinear irreducible characters  lie over $\theta$.

 \medskip

 It then follows that
  $$
  |\irr(U|\theta)|=|\irr(U/N|\theta)|=2.
  $$
In particular,  the characters in $\irr(U|\theta)$ are $Q_\theta$-invariant as $|Q_\theta|=3$. Thus, each character in $\irr(U|\theta)$ can be extendible to $I_G(\theta)$.  Furthermore,
$$
\{\eta(1)|\eta\in \irr(I_G(\theta)|\theta)\}=\{2\}.
$$
We conclude that
$$
|\irr_p(G|Z)|\leqslant \sum\limits_{\theta\in \irr(Z)-1_Z}|\irr_p(G|\theta)|=\sum\limits_{\theta\in \irr(Z)-1_Z}|\irr_p(I_G(\theta)|\theta)|=0.
$$

\textbf{Step 3.}\,  $|\irr_p(G)|=2$.

Recall that $K\cong Z_7\rtimes Z_3$ and thus $|\irr_p(K)|=2$. As $G/U\cong K$, we have that $|\irr_p(G/U)|=2$. As a result of the first two steps, we have that
$$
|\irr_p(G)|=|\irr_p(G/U)|+|\irr_p(G/Z|U/Z)|+|\irr_p(G|Z)|=2.
$$

\medskip

\textbf{Step 4.}\,  $l_p(G)= 3$.

Note that $N_G(P)\cap Z=\{(0,d)|d=0,1\}<Z$. Then we have that
$$
N_G(P)<N_G(P)Z.
$$
By the structure of $G/Z$ in \eqref{1051t}, we have that
$$
(N_G(P)Z\cap U)/Z=N_{G/Z}(PZ/Z)\cap U/Z<U/Z.
$$
This yields that  $N_G(P)Z<N_G(P)ZU=N_G(P)U$. In addition, we have that
$$
N_G(P)U/U=N_{G/U}(PU/U)=PU/U<G/U.
$$
Thus,
$$
N_G(P)<N_G(P)Z<N_G(P)U<G
$$
 is a strictly ascending chain of subgroups between $N_G(P)$ and $G$ with length $3$, which yields that $l_p(G)\ge 3$.
 
 \medskip

Now we show that $l_p(G)=3$. It is enough for us to show that
\begin{equation}\label{1700t}
l(N_G(P),N_G(P)Z)=l(N_G(P)Z,N_G(P)U)=l(N_G(P)U,G)=1.
\end{equation}

\medskip 

Suppose $L$ is a subgroup of $N_G(P)Z$ satisfying $N_G(P)<L$.  Then  $L\cap Z$ is a subgroup of $Z$ which is $P$-invariant.  It then follows that
$$
|L\cap Z|\equiv |C_{L\cap Z}(P)|=|C_Z(P)|=2 \mod\ 3.
$$
Thus,  $|L\cap Z|=5$ or $|L\cap Z|=8$ which yields that  $L\cap Z=Z$ as $L\cap Z$ is a subgroup of $Z$.  

\medskip 

In addition, we have 
$$
L=N_G(P)Z\cap L=N_G(P)(L\cap Z)=N_G(P)Z, 
$$
which implies that $l(N_G(P),N_G(P)Z)=1$.

\medskip 

Now we take $\bar{G}=G/Z$, $\bar{P}=PZ/Z$,  and $\bar{U}=U/Z$. Deducing as before we have
$$
l(N_G(P)Z,N_G(P)U)=l(N_G(P)Z/Z,N_G(P)U/Z)=l(N_{\bar{G}}(\bar{P}),N_{\bar{G}}(\bar{P})\bar{U})=1.
$$

Finally, noting that $|G:N_G(P)U|=7$ is a prime, we can get $l(N_G(P)U,G)=1$. Thus, equality \eqref{1700t} follows and the proof is complete.

 \eprf
 
\medskip

\section*{Acknowledgment}
The first author is supported by the Natural Science Foundation of China (No. 12371019). The authors are grateful to Professor Ping Jin for his valuable comments. 

\medskip 

\section*{Statement on the use of AI} 

During the preparation of this manuscript, the authors used OpenAI's ChatGPT (GPT-6). It gave the concrete counterexample, but it did not give a detailed or convincing explanation why the given example is a counterexample. The verification of this example was done by the authors based on the ideas mentioned at the beginning. The authors  are responsible for mathematical arguments, results, and final content of the paper.

\end{document}